\documentclass[reqno]{amsart}

\usepackage[utf8]{inputenc}
\usepackage[T1]{fontenc}

\usepackage{amsmath}
\usepackage{amssymb}
\usepackage{amsfonts}
\usepackage{graphicx}
\usepackage{amsthm}
\usepackage{enumerate}
\usepackage{lscape}
\usepackage{dsfont}
\usepackage{color}
\usepackage{mathtools}

\usepackage{setspace}
\newcommand{\Z}{\mathds{Z}}
              
\newcommand{\R}{\mathds{R}}

\newcommand{\Ric}{\mathrm{Ric}}

\newcommand{\C}{\mathds{C}}

\newcommand{\K}{K\"{a}hler}

\newcommand{\scal}{\operatorname{scal}}

\def\b{\beta}

\def\b1{{\rm id}}

\newtheorem{theor}{Theorem}[section]
\newtheorem{prop}[theor]{Proposition}

\newtheorem{example}{Example}
\newtheorem{rmk}{Remark}

\begin{document}

\title[On canonical radial metrics]{On canonical radial  K\"{a}hler metrics}

\author{Andrea Loi}
\address{(Andrea Loi) Dipartimento di Matematica \\
         Universit\`a di Cagliari, Via Ospedale 72, 09124  (Italy)}
         \email{loi@unica.it}

\author{Filippo Salis}
\address{(Filippo Salis) Dipartimento di Scienze Matematiche ``G. L. Lagrange'', Politecnico di Torino,
Corso Duca degli Abruzzi 24, 10129 Torino (Italy)}
\email{filippo.salis@polito.it}

\author{Fabio Zuddas}
\address{(Fabio Zuddas) Dipartimento di Matematica \\
         Universit\`a di Cagliari, Via Ospedale 72, 09124  (Italy)}
         \email{ fabio.zuddas@unica.it}

\thanks{
The first and the third authors were supported  by STAGE - Funded by Fondazione di Sardegna and by KASBA- Funded by Regione Autonoma della Sardegna. The second author was supported by PRIN 2017 ``Real and Complex Manifolds: topology, geometry and holomorphic dynamics” and MIUR grant ``Dipartimenti di Eccellenza 2018-2022”. All the three authors were supported by INdAM GNSAGA - Gruppo Nazionale per le Strutture Algebriche, Geometriche e le loro Applicazioni.
}

\subjclass[2000]{53C55, 32Q15, 32T15.} 
\keywords{\K\ \ metric, extremal metric; constant scalar curvature metric;  generalized constant scalar curvature metric; \K-Einstein metrics}

\begin{abstract}
We prove that   a radial  \K\ metric $g$  is \K-Einstein if and only if one of the following conditions is satisfied:
1. $g$ is  extremal and it is associated to a \K-Ricci soliton; 2. two different  generalized scalar curvatures of $g$ are constant;
3. $g$ is extremal  (not cscK) and one of its generalized scalar curvature is constant.
\end{abstract}
 
\maketitle

\tableofcontents  

\section{Introduction}
Given a complex manifold $M$ (compact or not)  it is an interesting and well-studied problem to
see when $M$ can be endowed with  some canonical metric.
Undoubtedly the most studied and important are the  \K-Einstein (KE) metrics.

Other prominent examples that generalize KE metrics and have attracted the attention of many mathematicians
are the following three types of \K\ metrics.

\vskip 0.1cm

\noindent
{\bf 1}. {\em Extremal
metrics}. Introduced  by Calabi \cite{calextrem},
are  those  metrics 
such that the (1,0)-part of the
Hamiltonian vector field associated to the scalar curvature  is
holomorphic.  The reader is referred to \cite{LSZext} and references therein for more details.
We denote by ${\mathcal Ext}(M)$ the set of 
extremal  metrics on $M$.

\vskip 0.1cm

\noindent
{\bf 2}. {\em The metrics associated to a 
\K-Ricci soliton (KRS)}. A KRS on a complex manifold $M$ is a pair $(g, X)$ consisting of a Kähler metric $g$ and a
holomorphic vector field $X$, called the {\em solitonic vector field}, such that 
\begin{equation}\label{eqkrsg}
\rho=\lambda \omega+L_{X}\omega
\end{equation}
for some $\lambda \in \mathbb{R}$, called the {\em solitonic constant}.
Here $\omega$ and $\rho$ are respectively the \K\ form and  the Ricci form of the metric $g$ 
and $L_X\omega$ denotes the Lie derivative of $\omega$ with respect to $X$.
KRS are  special solutions of the \K-Ricci flow 
and they generalize \K--Einstein (KE) metrics\footnote{For more information on KRS see  references in \cite{LM}.} . Indeed any
KE metric $g$ on a complex manifold $M$ gives rise to a
trivial KRS by choosing $X = 0$ or $X$ Killing with
respect to $g$. Obviously if the automorphism group of $M$ is
discrete then a \K--Ricci soliton $(g, X)$ is nothing but a
KE metric $g$. 
We denote by ${\mathcal K\mathcal R \mathcal S}(M)$ the set of  \K\ metrics $g$ on  $M$ such that  $(g, X)$ is a KRS, for some solitonic vector field  $X$.

\vskip 0.1cm

\noindent
{\bf 3}. {\em The $k$-generalized constant scalar curvature metrics,  $1\leq k\leq n$
(where $n$ is the complex dimension of $n$).}
Let $g$ be a \K\ metric. By definition, the  {\em $k$-generalized scalar curvature}, $1\leq k\leq n$,  $\rho_k(g)$ of $g$ 
are defined as (see \cite{ogiue}):
\begin{equation}\label{genscal}
\frac{\det \left(g_{i\bar j}+s\ \Ric_{i\bar j}\right)}{\det(g_{i\bar j})}=1+\sum_{k=1}^n\rho_k(g)s^k.\end{equation}
Notice that $\rho_1(g)=\scal_g$, where $\scal_g$ is the scalar curvature of the metric $g$.
Denote by ${\mathcal C}_k(M)$ the set of \K\ metrics $g$ on $M$ such that  $\rho_k(g)$  is a constant.

\vskip 0.3cm

For any complex manifold $M$ one clearly has the following inclusions:
\begin{equation}\label{inclusions}
{\mathcal K\mathcal R \mathcal S}(M)\supseteq{\mathcal K\mathcal E}(M)\subseteq{\mathcal C}_k(M),\  {\mathcal K\mathcal E}(M)\subseteq {\mathcal C}_1(M)\subseteq{\mathcal Ext}(M)
\end{equation}
where ${\mathcal K\mathcal E}(M)$ is the set of KE metrics on $M$.

It is then interesting to study  the following:

\vskip 0.3cm

\noindent
{\bf Problem.}
{\em Find conditions which ensure that a canonical \K\ metric of the types above is  KE.}

\vskip 0.3cm

In  this regard we recall some  results when $M$ is {\em compact},  summarized in the following 
theorem.

\vskip 0.3cm

\noindent
{\bf Theorem A.}
{\em Let $M$ be  a  compact complex  manifold $M$.  Then 
the following facts hold true.
\begin{itemize}
\item [(a)]
${\mathcal C}_k(M)\cap{\mathcal K\mathcal R \mathcal S}(M)\subseteq{\mathcal K\mathcal E}(M)$, for all $k\geq 1$.
\item [(b)] 
if $g\in {\mathcal Ext}(M)\cap {\mathcal K\mathcal R \mathcal S}(M)$ and assume that one of the two following 
conditions holds true:
\begin{itemize}
\item [(b1)]
$(M, g)$ is  toric;
\item [(b2)]
the   holomorphic sectional curvature of $g$ does not change sign.
\end{itemize}
Then $g$ is KE.
\end{itemize}}
\begin{proof}
Let $g$ be the \K\ metric associated to a KRS and $\omega$ its \K\ form.
Notice that  the solitonic vector field of a  KRS on a compact complex manifold is gradient and hence $\omega$
is  cohomologically Einstein. 
 Hence  (a) follows  by the first Corollary in \cite{ChOg} when $\rho_1(g)$ is constant 
 and when $\rho_k(g)$ is constant and different from zero, for $k\geq 1$.
 If $\rho_k(g)=0$ for $k\geq 1$ then \cite[Theorem 1]{ogiue} yields that $c_1(M)=0$, i.e. the KRS
 is steady and hence $g$ is  forced to be KE by \cite{CAOdef}.
 
The proofs of (b1) and (b2) can be found  in  \cite{CP17} and \cite{CP16} respectively.
\end{proof}

\begin{rmk}\rm
We do not know if  the assumptions (b1)  and (b2) can be dropped.
Notice that for the proof of (b2) one needs to use only that the KRS is gradient (always true in the compact case)
and the holomorphic sectional curvature does not change sign.
\end{rmk}

\begin{rmk}\rm
Notice that   the inclusion $ {\mathcal K\mathcal E}(M)\subseteq{\mathcal C}_k(M)\cap{\mathcal C}_1(M)$
(and hence the inclusion ${\mathcal K\mathcal E}(M)\subseteq{\mathcal C}_k(M)\cap {\mathcal Ext}(M)$) for  $k\geq 1$ is strict for a compact complex manifold $M$ even if one assumes (b2) in Theorem A. Indeed the metric $g$ given by the  product of the flat metric  and the Fubini-Study metric  on $T^{n-k+1}\times \C P^{k-1}$ (where $T^{n-k+1}$ is the complex torus and $\C P^{k-1}$ the complex projective space) has constant scalar curvature, $\rho_k(g)=0$, it is not KE  and its holomorphic sectional curvature is non-negative (cfr. the final Remark in  \cite{ChOg}).
In light of  (b1) in Theorem A  it could be interesting to see if the equality   ${\mathcal C}_k(M)\cap{\mathcal Ext}(M)= {\mathcal K\mathcal E}(M)$ holds true  in the compact toric case.
\end{rmk}

When the manifold involved is noncompact 
the previous  problem has been studied by the first and third author of the present paper
for Hartogs domains.
More precisely in \cite{osaka} it is shown  that if  the \K\ metric $g$ naturally associated to an Hartogs domain $D\subset \C^n$
belongs to one of the three types described above then $g$ is forced to be KE (and hence $(D, g)$
is  holomorphically isometric to an open subset of the complex hyperbolic  $n$-space).

In this paper we restrict to radial metrics, namely those \K\ metrics 
$g$ on (noncompact) complex manifolds   which admit a global K\"ahler potential which depends only on the sum $|z|^2 = |z_1|^2 + \cdots + |z_n|^2$ of the local coordinates' moduli. 

If  $M$ is  a  complex manifold we denote by
$${\mathcal Rad}(M)=\{\mbox{radial \K\ metrics on}\  M\}$$

 The main result of the paper is the following theorem
 which shows  in particular that  in the noncompact  radial case
 the same conclusion of Theorem A can be achieved without any assumption 
 on the curvature of the metric.

\begin{theor}\label{mainteor}
Let $M$ be a complex manifold.
Then the following facts hold true.
\begin{itemize}
\item [(i)]
${\mathcal Ext}(M)\cap {\mathcal K\mathcal R \mathcal S}(M)\cap {\mathcal Rad}(M)={\mathcal K\mathcal E}(M)\cap {\mathcal Rad}(M)$;
\item [(ii)]
${\mathcal C}_k(M)\cap {\mathcal C}_h(M)\cap {\mathcal Rad}(M)={\mathcal K\mathcal E}(M)\cap {\mathcal Rad}(M), \forall h, k\geq 1, h\neq k$;
\item [(iii)]
${\mathcal C}_k(M)\cap {\mathcal Ext}(M)\cap {\mathcal Rad}(M)={\mathcal K\mathcal E}(M)\cap {\mathcal Rad}(M), \forall k> 1$;
\item [(iv)]
${\mathcal C}_k(M)\cap{\mathcal K\mathcal R \mathcal S}(M)\cap {\mathcal Rad}(M)={\mathcal K\mathcal E}(M)\cap {\mathcal Rad}(M), \forall k\geq 1$.
\end{itemize}
\end{theor}

In the next section we collect some results on radial metrics and we prove  Theorem \ref{mainteor}.
In the final section we provide some explicit examples and compare Theorem \ref{mainteor} with Theorem A.

\section{Radial canonical \K\ metrics}
Let $g$ be a radial  \K\ metric on a connected  complex manifold $M$, equipped with complex coordinates $z_1, \dots ,z_n$ and let $\omega$ and $\rho$ be respectively 
the \K\ form and the Ricci form associated to $g$.
Then  there exists a   smooth  function 
$$f: (r_{\inf}, r_{\sup})\rightarrow \R, \ 0\leq r_{\inf}<r_{\sup}\leq\infty,$$
where $(r_{\inf}, r_{\sup})$ is the maximal domain where $f(r)$ is defined 
such that
 \begin{equation}\label{omegar}
 \omega =\frac{i}{2} \partial \bar \partial f(r), \ r=|z|^2=|z_1|^2+\cdots +|z_n|^2,
 \end{equation}
 i.e. $f(r)$ is a radial potential for the metric $g$.

One can easily see that 
the matrix of the metric $g$ and of the Ricci form $\rho$ read  as
\begin{equation}\label{metric}
\omega_{i\bar j}=f'(r)\delta_{ij}+f''(r) \bar z_i z_j.
\end{equation}
\begin{equation}\label{formaricci}
\rho_{i\bar j}= L'(r) \delta_{i j} + L''(r) \bar z_ i z_j,
\end{equation}
where $L(r) = - \log(\det g)(r)$.

Set 
\begin{equation}\label{y(r)}
y(r):=rf'(r).
\end{equation}
and
\begin{equation}\label{psi(y)}
\psi(r) := ry'(r).
\end{equation}
Then 
\begin{equation}\label{psi(y)2}
\psi(r) = \frac{dy}{dt},\  r=e^t.
\end{equation}

The fact that $g$ is a metric is equivalent to 
$y(r)>0$ and $\psi (r)>0$, $\forall r\in (r_{\inf}, r_{\sup})$.
Then 
\begin{equation}\label{limtr}
\lim_{r\rightarrow r^+_{\inf}}y(r)= y_{\inf}
\end{equation}
is a non negative real number.
Similarly set
\begin{equation}\label{limtrbis}
\lim_{r\rightarrow r^-_{\sup}}y(r)= y_{\sup}\in (0, +\infty].
\end{equation}

Therefore we can invert the map
$$(r_{\inf}, r_{\sup})\rightarrow  (y_{\inf}, y_{\sup}),\  r\mapsto y(r)=rf'(r)$$ 
on $(r_{\inf}, r_{\sup})$ and think $r$ as a function of $y$, i.e. $r=r(y)$.

Hence we can set
\begin{equation}\label{psi}
\psi (y):=\psi (r(y)).
\end{equation}

Finally,  from (\ref{metric}), we easily get 
\begin{equation}\label{detmetric}
(\det g_{i\bar j})(r)=\frac{(y(r))^{n-1}\psi (y(r))}{r^{n}}.
\end{equation}

The following three propositions  (Proposition \ref{lemmasimple}, Proposition \ref{mainprop} and Proposition \ref{mainprop2})
are the key tools for the proof of Theorem \ref{mainteor} and 
provide us with the  explicit expressions of  radial extremal metrics,  radial KRS and radial generalized cscK metrics respectively, in terms of the  functions 
$y$ and $\psi (y)$ defined by \eqref{y(r)} and \eqref{psi(y)}.

\begin{prop}\label{lemmasimple}
A radial \K\ metric  $g$ is extremal if and only if
\begin{equation}\label{ypsi}
\psi(y) = y - \frac{A}{y^{n-1}} - \frac{B}{y^{n-2}} - C y^2 - D y^3. 
\end{equation}
for some $A,B,C,D \in \R$.
Moreover, 
\begin{itemize}
\item [(a)]
if $n=1$, $g$ is  KE (i.e. a complex space form) iff  $D=0$. Moreover, its (constant) scalar curvature is given by $2C$;
\item  [(b)]
if $n\geq 2$, $g$ is  KE
iff $B=D=0$ with Einstein constant $2C(n+1)$. Moreover, the metric is flat iff $A=B=C=D=0$.
\end{itemize}
\end{prop}
\begin{proof}
See  \cite[Lemma 2.1] {LSZext} for a proof.
\end{proof}

\begin{rmk}\rm\label{rmkeinstein} 
From (\ref{ypsi}) we easily deduce that if a \K-Einstein metric is defined at the origin $r=0$ then it is a complex space form. Indeed, the metric is Einstein if and only if 
$$\psi(y) = y - \frac{A}{y^{n-1}} - C y^2,$$ 
which immediately implies that $A=0$ if the metric is defined at the origin since in that case $y(r) = rf'(r) =0$  and $\psi(r) = r(rf'(r))'=0$ for $r=0$.
\end{rmk}

 \begin{prop}\label{mainprop}
Let $g$ be a radial KRS with solitonic constant $\lambda$.
Then the following facts hold true.

If $n=1$ then there exist  $\mu, k\in\R $ such that

\begin{equation}\label{finalen=1}
\dot\psi(y)= \mu \psi(y) + k+1 -\lambda y
\end{equation}

and if $\mu=0$ then the  soliton is trivial (i.e. a complex space form).
If $\mu\neq 0$ then
\begin{equation}\label{psiespln=12}
\psi(y) = \nu e^{\mu y} + \frac{\lambda}{\mu} y + \left( \frac{\lambda}{\mu^2} - \frac{k + 1}{\mu} \right)
\end{equation}
and  the soliton is trivial iff it is flat iff $\nu=0$.

If $n\geq 2$ then there exists $\mu\in\R$ such that 

\begin{equation}\label{finale9}
\dot\psi(y) = \left( \mu - \frac{n-1}{y} \right) \psi (y) + n - \lambda y 
\end{equation}
and if $\mu=0$ the soliton is trivial  (i.e. KE).
If $\mu\neq 0$ then
\begin{equation}\label{psiesplicita}
 \psi(y) = \frac{\nu e^{\mu y}}{y^{n-1}} + \frac{\lambda}{\mu} y +  \frac{\lambda - \mu}{\mu^{1+n}} \sum_{j=0}^{n-1} \frac{n!}{j!} \mu^j y^{j+1-n}.
\end{equation}
and  the soliton is trivial iff it is flat iff $\nu=0$ and  $\mu =\lambda$.
\end{prop}
\begin{proof}
See  either \cite[Proposition 2.2]{LSZkrs} or  \cite{FIK} for a proof.
\end{proof}

\begin{prop}\label{mainprop2}
Let $g$ be a radial \K\ metric and set
\begin{equation}\label{sigmafond}
\sigma(y):=\frac{1}{y^{n-1}}\frac{d}{dy}\left[y^{n-1}\psi (y)\right]=\dot\psi(y)+\frac{(n-1)\psi(y)}{y}.
\end{equation}
Then its $k$-th generalized scalar curvature $\rho_k(g)$, $1\leq k\leq n$, is constant, i.e.  $\rho_k(g)=\rho_k$, if and only if
\begin{equation}\label{rho}
\sigma(y)=n-y\left(A_k+\frac{B_k}{y^{n}}\right)^{1/k},
\end{equation}
where $A_k=\rho_k\frac{k!(n-k)!}{n!}$ and $B_k$ is  constant (depending on $k$).

Moreover, $g$ is KE with Einstein constant $\lambda$ if and only if $\sigma(y) = n - \frac{\lambda}{2} y$.
\end{prop}
\begin{proof}
Let $g$ be radial with \K\ potential $f(r)$, where $r = |z_1|^2+\mathellipsis+|z_n|^2$. By (\ref{y(r)}) and (\ref{psi(y)}) we immediately get $f'(r) = \frac{y(r)}{r}$ and $f''(r) = \frac{\psi(y(r)) - y(r)}{r^2}$, which combined with  (\ref{metric}) yields 
$$g_{i\bar j}(r)=\frac{\psi(y(r))-y(r)}{r^{2}}\bar z_i z_j+\frac{y(r)}{r}\delta_{ij},$$
Also by (\ref{detmetric}) we have
$$L(r) = - \log (\det g)(r) = -(n-1) \log y(r) - \log \psi(y(r)) + n \log r$$
and then, by using $y'(r) = \frac{\psi(y(r))}{r}$ and (\ref{sigmafond})
$$L'(r) = -\frac{n-1}{y} \frac{\psi}{r} - \frac{\dot\psi(y)}{r} + \frac{n}{r} = \frac{n - \sigma(y(r))}{r}$$
By (\ref{formaricci}),  then finally one gets (cfr. also \cite{LSZ}) 
$$\Ric_{i\bar j}(r)=\frac{-\dot\sigma(y(r))\psi(y(r))+\sigma(y(r))-n}{r^{2}}\bar z_i z_j+\frac{n-\sigma(y(r))}{r}\delta_{ij}.$$
Then \eqref{genscal} reads as
\begin{multline}\label{genscalrad}
1+\sum_{k=1}^n\rho_ks^k=\frac{(\psi(y)-s\dot\sigma (y)\psi(y))(y+sn-s\sigma(y))^{n-1}}{\psi(y) y^{n-1}}=\left(1-s\dot\sigma(y)\right)\left(1+s\frac{n-\sigma(y)}{y}\right)^{n-1}=\\
=1+\sum_{k=1}^{n-1}\left(\frac{n-\sigma(y)}{y}\right)^{k-1}\left[\binom{n-1}{k}\frac{n-\sigma(y)}{y}-\binom{n-1}{k-1}\dot\sigma (y)\right]s^k-\dot\sigma (y)\left(\frac{n-\sigma(y)}{y}\right)^{n-1} s^n.
\end{multline}
Then $n$-th generalized scalar curvature $\rho_n$ is constant  if and only if  
$$-\dot\sigma (y)\left({n-\sigma(y)}\right)^{n-1}=\rho_n\ y^{n-1}$$ 
which integrates to 
\begin{equation*}
\sigma(y)=n- y\left(\rho_n\ +\frac{B_n}{y^n}\right)^{1/n},\end{equation*}\\
i.e. \eqref{rho} for $k=n$.

On the other hand the $k$-th generalized scalar curvature  $\rho_k$,  $1\leq k\leq n-1$,    is  constant  if and only if
\begin{equation}\label{Rk}
R_ky^{k}-(n-k)(n-\sigma(y))^{k}+ky\dot\sigma(y)(n-\sigma(y))^{k-1}=0,
\end{equation}
where $R_k=\rho_k\frac{k!(n-k)!}{(n-1)!}$, 

If $R_ky^{k}-n(n-\sigma(y))^{k}=0$
then 
\begin{equation*}
\sigma(y)=n-\left(\frac{R_k}{n}\right)^{1/k} y.\end{equation*}
i.e. \eqref{rho} with $B_k=0$.

If  $R_ky^{k}-n(n-\sigma(y))^{k}\neq 0$
then \eqref{Rk} gives
$$\frac{n-k}{y}+\frac{R_k k y^{k-1}+kn(n-\sigma(y))^{k-1}\dot\sigma(y)}{R_ky^{k}-n(n-\sigma(y))^{k}}=0$$
which integrates to 
\begin{equation*}
\sigma(y)=n-\left(\frac{R_k}{n}y^{k}+B_ky^{k-n}\right)^{1/k}.
\end{equation*}

For the last assertion of the proposition we assume $n\geq 2$ (the case $n=1$ is obtained similarly). We know by  Proposition \ref{lemmasimple} that the metric is KE with Einstein constant $\lambda$ if and only if $\psi(y) = y - \frac{\lambda}{2(n+1)} y^2 - \frac{A}{y^{n-1}}$. It is immediate to see that this is equivalent to $\frac{1}{y^{n-1}}\frac{d}{dy}\left[y^{n-1}\psi (y)\right] = n - \frac{\lambda}{2}y$, which by (\ref{sigmafond}) proves the assertion.
\end{proof}

\begin{rmk}\label{ke}\rm
Equation \eqref{rho} combined with \eqref{sigmafond}, together with a choice of initial values $y_0 > 0$ and $\psi(y_0) > 0$, yield a Cauchy problem for $\psi(y)$ whose solution is a $k$-generalized cscK which is not cscK.
For an explicit example, take $k=n$ and $A_n =0$ in \eqref{rho}: then $\sigma(y) = c:= n - (B_n)^{1/n}$ which by \eqref{sigmafond} yields $\psi(y) = \frac{c}{n} y + \frac{d}{y^{n-1}}$.
Notice that if either $n\neq 1$  or $c \neq n$, i.e. $B_n \neq 0$, this is not an extremal metric. In particular, for $d=0$, by $\psi(y) = \frac{dy}{dt}$ and by recalling that $r = e^t$ and $y(r) = rf'(r)$, one gets the potential $f(r) =\beta r^{c/n}$, for some $\beta \in \R$.
\end{rmk}

We are now in the position to prove Theorem \ref{mainteor}.

\begin{proof}[Proof  of Theorem \ref{mainteor}]
To show (i)
let us assume that a radial metric is both extremal and KRS.
Let us distinguish the cases $n=1$ and $n \geq 2$.

If $n=1$ the extremal condition \eqref{ypsi} and its derivative  read as
\begin{equation}\label{extn=1}
\psi (y)=(1-B) y-A- C y^2 -D y^3.
\end{equation}
\begin{equation}\label{derextn=1}
\dot\psi(y)=(1-B)-2 C y -3D y^2.
\end{equation}
By inserting \eqref{extn=1} into the soliton equation  \eqref{finalen=1} (for $n=1$)
we get 
 $$\dot\psi(y)=\left[\mu(1-B)-\lambda\right]y-\mu Cy^2-\mu Dy^3 -\mu A+k+1, $$
 which compared with \eqref{derextn=1} forces the coefficient of $y^3$ to vanish, i.e. $\mu D=0$.
If  $D= 0$ or $\mu=0$ the metric is KE respectively by Proposition \ref{lemmasimple} and Proposition \ref{mainprop}.
Let us now assume $n\geq 2$. By inserting equation for extremal metrics  (\ref{ypsi}) into the soliton equation (\ref{finale9}) we obtain
$$
\ \ \ \ \ \ \ \ \ \ \ \ \ \ \ \ \ \ \dot\psi(y)= 1 +\left [C(n-1) +\mu - \lambda\right]y + \left[D(n-1) - C \mu\right] y^2 - D \mu y^3  +$$
$$+\frac{A(n-1)}{y^{n}} + \frac{B(n-1) - \mu A}{y^{n-1}} -\frac{B\mu}{y^{n-2}} 
$$
On the other hand, derivating (\ref{ypsi}) we get

$$\dot\psi(y)= 1 - \frac{A(1-n)}{y^{n}} - \frac{B(2-n)}{y^{n-1}} - 2Cy - 3D y^2$$

Comparing these two last expressions and  observing that in the first one there are  the terms in $y^3$ and $\frac{1}{y^{n-2}}$ which are not in the second one, one finds either $\mu=0$ and  then  the soliton is trivial by Proposition \ref{mainprop}, or $B=D=0$, which by Proposition \ref{lemmasimple},  again implies  that the  metric is KE. Hence (i) is proved.

In order to prove (ii), assume that the generalized curvatures $\rho_k(g)$ and $\rho_h(g)$ are constant for some $h , k \geq 1$, $h \neq k$. By (\ref{rho}) in Proposition \ref{mainprop2}, we must have
$$\left(A_k+\frac{B_k}{y^{n}}\right)^{1/k} = \left(A_h+\frac{B_h}{y^{n}}\right)^{1/h}$$
which clearly implies that $B_k=B_h=0$ and $(A_k)^{1/k} = (A_h)^{1/h} = A$.

Then, $\sigma(y) = n - Ay$ and the metric is KE by the last assertion of Proposition \ref{mainprop2}.

We now prove (iii).
If a radial  \K\ metric $g$  is extremal then by combining  \eqref{ypsi} and \eqref{sigmafond}
one gets:
\begin{equation}\label{sigmaextr}
\sigma(y)= n-\frac{B}{ y^{n-1}}-C (n+1) y -D (n+2) y^2.   
\end{equation}
Assume that the $k$-th generalized scalar curvature $\rho_k(g)$ (with $k>1$) is constant: by Proposition \ref{mainprop2} and by comparing (\ref{rho})  with (\ref{sigmaextr}) we see that $\left(A_k+\frac{B_k}{y^{n}}\right)^{1/k}$ must be a rational function. This is possible only if either $B_k=0$ (and hence the metric is KE by Proposition \ref{mainprop2}) or $A_k=0$ and $\frac{n}{k} \in \Z$. In the latter 
$\sigma(y) = n - \frac{(B_k)^{1/k}}{y^{\frac{n}{k} -1}}$ which compared with (\ref{sigmaextr}) and recalling that $k>1$ yields again $B_k=0$.

Finally we prove (iv). 
By the equations \eqref{psiespln=12} and \eqref{psiesplicita} of a radial non trivial KRS one easily gets
that  \eqref{sigmafond} reads as 
\begin{equation}\label{sigma}
\sigma(y)=\frac{\mu\nu e^{\mu y}}{y^{n-1}} + n \frac{\lambda}{\mu} +\frac{\lambda-\mu}{\mu^{1+n}}\sum_{j=1}^{n-1}\frac{n!}{(j-1)!} \mu^j y^{n-j}.\end{equation}
By comparing the previous equation with \eqref{rho}, we easily get that if a  radial non trivial  KRS has  constant $k$-th generalized scalar curvature (with $1\leq k\leq n$), then $\nu=0$ and $\lambda = \mu$, which by  the last assertion of Proposition \ref{mainprop} means that $g$  is KE
(actually Ricci flat), yielding the desired contradiction and proving (iv). 
\end{proof}

\section{Some final remarks}
The assertion (i) in Theorem \ref{mainteor} should be compared with  (b2) of Theorem A in the introduction. 
Hence it is worth to exhibit  radial extremal metrics and non trivial radial KRS  with sign-changing holomorphic sectional curvature. 
This is done in the following two examples.
We first recall that in \cite{LSZ2} we have shown that, given a radial metric, in the point $p = (z_1, 0, \dots, 0)$ the only non vanishing components of the Riemann tensor $R_{i \bar j k \bar l}$ are

$$R_{1 \bar 1 1 \bar 1} = \frac{\ddot\psi(y) \psi^2(y)}{r^2}$$
$$R_{1 \bar 1 i \bar i} = \frac{\dot\psi(y)y - \psi(y)}{y r^2}$$
$$R_{i \bar i i \bar i} = 2 R_{i \bar i j \bar j} = 2 \frac{\psi(y)- y }{r^2}$$

Then, in $p$, the holomorphic sectional curvature along $Z = \sum_k \xi_k \frac{\partial}{\partial z_k}$ is

\vskip0.3cm

$$R(Z, \bar Z, Z, \bar Z) = \frac{\ddot\psi(y) \psi^2(y)}{r^2} |\xi_1|^4 + \frac{\dot\psi(y) y - \psi(y)}{y r^2} |\xi_1|^2 \sum |\xi_i|^2 + $$

$$\ \ \ \ \ \ \  \ \ \ \ \ +\frac{\psi(y) - y }{r^2} \sum |\xi_i|^2 |\xi_j|^2 + 2 \frac{\psi(y) - y }{r^2} \sum |\xi_i|^4 .$$

If we assume that $\xi_2 = \cdots = \xi_n = 0$ this formula (always in $p$) reduces to 

\begin{equation}\label{Rsemplif}
R(Z, \bar Z, Z, \bar Z) = \frac{\ddot\psi(y) \psi^2(y)}{r^2} |\xi_1|^4
\end{equation}
Thus to  find  radial extremal metrics or radial KRS
with sign-changing holomorphic sectional curvature, it will be enough to find  metrics for which $\ddot\psi$ changes sign in its domain of definition.

\begin{example} \rm 
Take the radial extremal metric in  dimension $n \geq 2$ with  $A=B=0$, $C=1$, $D=-1$ in (\ref{ypsi}), i.e.
$$\psi(y) = y -y^2 + y^3.$$ 

Since $\ddot\psi(y) = -2 + 6y$ we have that $\ddot\psi(y)$ changes sign in a neighbourhood of $y = \frac{1}{3}$; moreover, being $\psi(\frac{1}{3}) > 0$ the local solution $y(t)$ of the Cauchy problem $\frac{dy}{dt} = \psi(y(t)), \ y(t_0) = \frac{1}{3}$, for any $t_0 \in \R$, satisfies the conditions $y > 0$ and $\psi(y) >0$ to represent a metric and then $\psi$ defines an extremal metric which, by (\ref{Rsemplif}), has sign-changing holomorphic sectional curvature.
\end{example}

\begin{example} \rm 
In order to find a nontrivial radial KRS  with sign-changing holomorphic sectional curvature, take for example
$$n=3, \ \ \nu=0, \ \ \mu < 0, \ \  \lambda < -\frac{5}{4} \mu,\  \lambda \neq \mu$$
in \eqref{psiesplicita}, i.e.

\begin{equation}\label{psiesplicitaANTICIPATAsegno}
\psi(y) = \frac{\lambda}{\mu} y + \frac{\lambda - \mu}{\mu^4} \left( \frac{6}{y^2} + \frac{6 \mu}{y} + 3 \mu^2 \right)
\end{equation}

One gets

\begin{equation}\label{psiesplicitaANTICIPATAsegno2}
\ddot\psi(y) = \frac{12(\lambda - \mu)}{\mu^4 y^4} (3 + \mu y)
\end{equation}

and then $\ddot\psi(y)$ changes sign in a neighbourhood of $y = - \frac{3}{\mu}$, which is positive by the assumptions. Moreover, one finds

$$\psi \left( - \frac{3}{\mu} \right) = \frac{-4 \lambda - 5 \mu}{3\mu^2}$$

which is  positive by the assumptions. 
Then, we conclude as in the previous example that $\psi$ yields a non-trivial Ricci soliton which, by (\ref{Rsemplif}), has sign-changing holomorphic sectional curvature (the non-triviality is guaranteed by $\lambda \neq \mu$).
\end{example}

\end{document}